\documentclass{amsart}

\usepackage{amssymb}
\usepackage{mathrsfs}

\newtheorem{thm}{Theorem}[section]

\newtheorem{lem}[thm]{Lemma}
\newtheorem*{lem*}{Lemma}

\newtheorem{prop}[thm]{Proposition}

\theoremstyle{remark}

\theoremstyle{definition}

\newcounter{my_enumerate_counter}

\newcommand\comment[1]{}

\renewcommand{\epsilon}{\varepsilon}

\title[von Neumann-Day problem]
{A geometric solution to the von {N}eumann-Day 
problem for finitely presented groups}



\title[Tarski numbers and piecewise projective groups]
{An upper bound for the Tarski numbers of non amenable groups of piecewise projective homeomorphisms}


\thanks{
The author thanks Nicolas Monod and Justin Moore for helpful discussions and comments.}

\author{Yash Lodha}

\address{EPFL
\\ Lausanne\\ Switzerland}

\email{{\tt yash.lodha@epfl.ch}}

\begin{document}

\begin{abstract}
The Tarski number of a non amenable group is the smallest number of pieces needed for a paradoxical decomposition of the group.
Non amenable groups of piecewise projective homeomorphisms were introduced in \cite{Monod}, and non amenable finitely presented groups of piecewise projective homeomorphisms were introduced
in \cite{LodhaMoore}.
These groups do not contain non abelian free subgroups.
In this article we prove that the Tarski number of all groups in both families is at most $25$.
In particular we demonstrate the existence of a paradoxical decomposition with $25$ pieces.

Our argument also applies to any group of piecewise projective homeomorphisms 
that contains as a subgroup the group of piecewise $PSL_2(\mathbb{Z})$ homeomorphisms of $\mathbb{R}$
with rational breakpoints, and an affine map that is a not an integer translation.
\end{abstract}

\maketitle

The von Neumann-Day problem is a striking issue that lurks around the boundary between amenability and non amenability for discrete groups.
The question asks whether there exist non amenable groups that do not contain non abelian free subgroups.
It was solved by Olshanskii around $1980$ (see \cite{Olsh}).
In $2012$ Monod observed in \cite{Monod} that the group of piecewise projective homeomorphisms of the real line is a counterexample.
In his article, Monod produces an infinite family of examples, some of which are countable, but none are finitely presentable.
In \cite{LodhaMoore} we constructed finitely presented non amenable subgroups of Monod's group $H(\mathbb{Z}[\frac{1}{\sqrt{2}}])$.
A striking feature of these examples is that they admit short finite presentations.

Every non amenable group $G$ admits a paradoxical decomposition.
In particular, there exist the following:
\begin{enumerate}
\item A partition of $G$ into pairwise disjoint subsets $S_1,...,S_n,T_1,...,T_m$.
\item Elements $s_1,...,s_n,t_1,...,t_m\in G$.
\end{enumerate}
which satisfy that $G=\bigcup_{1\leq i\leq n}s_iS_i=\bigcup_{1\leq i\leq m}t_iT_i$.
The number of pieces of this paradoxical decompsition is defined as the number $n+m$. 
The Tarski number of a non amenable group is defined as the smallest number of pieces needed for a paradoxical decomposition of the group.

It was observed by von Neumann that the free group of rank $2$ admits a paradoxical decomposition with four pieces.
It follows that any group that contains a non abelian free subgroup admits a paradoxical decomposition with four pieces.
By a theorem of Dekker and Jonsson (see Proposition $20$ in \cite{CeccGrigHarpe}), the converse holds as well.
Therefore, a group has Tarski number $4$ if and only if it contains non abelian free subgroups. 

The systematic study of Tarski numbers of groups was initiated in \cite{CeccGrigHarpe}.
It was proved that the Tarski number of any torsion group is at least six,
and that the Tarski number of any non-cyclic free Burnside group of odd exponent at least $665$
is between $6$ and $14$.

More recently, it was shown in \cite{ErshovGolanSapir} that finitely generated non amenable groups can have arbitrarily large Tarski numbers.
The authors also construct explicit examples of groups with Tarski numbers $5$ and $6$.
These examples are respectively torsion-by-abelian and torsion groups.

In this article we establish an upper bound for Tarski numbers
in the torsion free landscape of groups of piecewise projective homeomorphisms of the real line. 

\begin{thm}\label{main1}
All groups in Monod's family $\{H(A)\mid A\text{ is a dense subring of }\mathbb{R}\}$ defined in \cite{Monod} have Tarski number at most $25$.
Both finitely presented nonamenable groups defined in \cite{LodhaMoore} have Tarski number at most $25$.
\end{thm}

In fact our argument applies to other groups of piecewise projective homeomorphisms of the real line,
which do not contain free subgroups.

\begin{thm}\label{main2}
Let $F$ be the group of piecewise $PSL_2(\mathbb{Z})$ homeomorphisms of $\mathbb{R}$ with breakpoints in the set of rational numbers.
Let $f(t)=at+r$ such that $a,r\in \mathbb{R}, a>0$ and if $a=1$, then $r\in \mathbb{R}\setminus \mathbb{Z}$. 
Then the group $\langle F,f\rangle$ admits a paradoxical decomposition with $25$ pieces.
\end{thm}

It follows that any group of homeomorphisms of the real line containing $F$ (as in Theorem \ref{main2}) and some affine map that is not equal to an integer translation admits a paradoxical decomposition
with $25$ pieces. 
The group $F$ above is isomorphic to Thompson's group $F$,
which is usually described as a group of piecewise linear homeomorphisms of the unit interval.
(The piecewise projective version of $F$ was introduced by Thurston in the $1970$s.
See \cite{LodhaMoore} for more details.)
 
 \section{Preliminaries}
The Tarski number of a group is bounded above by Tarski numbers of its subgroups,
so it suffices to prove the main theorem for any countable group in the family. 
The actions of the groups on $\mathbb{R}$ will be right actions, however the actions of the groups on themselves will be left actions.
We shall fix $\mu$ as the Lebesgue measure on the real line.
If a group $G$ acts on a set $X$, we denote by $E_G^X$ the associated orbit equivalence relation.
We shall use the following basic fact about the orbit equivalence relations of our groups.\\

\begin{lem} 
Let $H$ be any countable group in the family of groups considered in theorems \ref{main1} and \ref{main2}. 
Then there is a proper overgroup $\mathbf{M}$ of $PSL_2(\mathbb{Z})$ in $PSL_2(\mathbb{R})$ such that 
$E_{\mathbf{M}}^{\mathbb{R}}=E_H^{\mathbb{R}}$.
\end{lem}

\begin{proof}
In Proposition $9$ of \cite{Monod} it has been shown that $E_{H(A)}^{\mathbb{R}}=E_{PSL_2(A)}^{\mathbb{R}}$.
Recall that in \cite{LodhaMoore} we construct two finitely presented non amenable groups $G_0$ and $G$.  
(In particular it was shown that $G_0$ admits a presentation with $3$ generators and $9$ relations, and is a subgroup of $G$.)
In \cite{LodhaMoore} it was also shown that $E_{G_0}^{\mathbb{R}}=E_{\langle PSL_2(\mathbb{Z}), \Gamma\rangle}^{\mathbb{R}}$
where $\Gamma=\left( \begin{array}{cc}
\sqrt{2} & 0  \\
0 & \frac{1}{\sqrt{2}}  \end{array} \right)$.
It is not too hard to check that in fact $E_G^{\mathbb{R}}=E_{G_0}^{\mathbb{R}}$.

For the groups of the form $\langle F,f\rangle$ described in Theorem \ref{main2}, we claim that 
$E_{\langle F,f\rangle}^{\mathbb{R}}=E_{\langle PSL_2(\mathbb{Z}),\Gamma \rangle}^{\mathbb{R}}$
where $\Gamma=\left( \begin{array}{cc}
a & b  \\
0 & a^{-1}  \end{array} \right)$
for some $a,b\in \mathbb{R}$.
This easily follows from the fact that $E_{F}^{\mathbb{R}}=E_{PSL_2(\mathbb{Z})}^{\mathbb{R}}$
(this is established in the last paragraph of page $4$ in \cite{LodhaMoore}), together with the existence of the globally defined affine map.
Since the affine map is not an integer translation, we conclude that $\langle PSL_2(\mathbb{Z}),\Gamma \rangle$
is a proper overgroup of $PSL_2(\mathbb{Z})$.
\end{proof}

For the rest of the article we fix $H$ as any countable group in our family and $\mathbf{M}$ as above.
From Th\'{e}or\`{e}me $3$ from \cite{GhysCarriere} it follows that $\mathbf{M}$ contains a non discrete free subgroup $\bf{F}$,
which we also fix throughout the article.
We endow $\mathbf{F}$ with the metric induced by $PSL_2(\mathbb{R})$ and denote by $B_{\delta}(1_{\bf{F}})$
 the ball of radius $\delta$ around the identity element.

Let $\mathbf{F}$ be freely generated by $a,b$.
Consider the set of right cosets of $\mathbf{F}$ in $\mathbf{M}$,
and fix $S$ as a set of coset representatives.
For $c\in \{a,b\}$ we define $\mathbf{F}_c$ as the set containing the empty word if $c=a$, together with all reduced words beginning with $c$ or $c^{-1}$.
Clearly $\mathbf{F}_a,\mathbf{F}_b$ provide a partition of $\mathbf{F}$ into disjoint sets.
Next, define $P_c=\{w s\mid w\in \mathbf{F}_c, s\in S\}$.
Clearly, $P_a, P_b$ provide a disjoint partition of $\mathbf{M}$.

Let $P_f(H)$ be the set of finite subsets of $H$.
For each $u\in P_f(H)$ we also denote by $u$ the probability measure on $H$ which assigns measure $\frac{1}{|u|}$
to each element of $u$.
For $x\in \mathbb{R}$, we denote by $[x]$ the orbit of $x$, and define $\lambda_{u,x}:2^{[x]}\to [0,1]$
as $\lambda_{u,x}(A)=u(\{g\in H\mid x\cdot g\in A\})$.
Clearly, this is a measurable assignment of measures on the orbits.
This means that for every measurable set $E\subset E_{\mathbf{M}}^{\mathbb{R}}$,
the function $x\mapsto \lambda_{u,x}(\{y\in \mathbb{R}\mid (x,y)\in E\})$
is measurable.
We now define a measurable partition of the equivalence relation $E_{\mathbf{M}}^{\mathbb{R}}$.

For each $x\in \mathbb{R}$, define $A_x=\{y\in [x]\mid \exists g\in P_a, x\cdot g=y\}=x\cdot P_a$
and $B_x=\{y\in [x]\mid \exists g\in P_b, x\cdot g=y\}=x\cdot P_b$.
Note that $A_x\cup B_x=[x]$ and $A_x\cap B_x=\emptyset$ almost everywhere.
The latter is a direct consequence of the fact that the action of $\mathbf{M}$ on $\mathbb{R}$ is a.e. free.

Further, define $\phi_{u,a}: \mathbb{R}\to [0,1]$ as $\phi_{u,a}(x)=\lambda_{u,x}(A_x)$
and $\phi_{u,b}: \mathbb{R}\to [0,1]$ as $\phi_{u,b}(x)=\lambda_{u,x}(B_x)$. 
We remark that $\phi_{u,a}(x)+\phi_{u,b}(x)=1$ for $x\in \mathbb{R}$ almost everywhere.
Clearly, these are measurable functions since the assignment of measures on the orbits is measurable.

We will require the following Lemma about the continuous action of $\bf{F}$ on $\mathbb{R}\cup \{\infty\}$
which follows from the fact that derivatives of projective transformations vary continuously with respect to the usual topology on $PSL_2(\mathbb{R})$.

\begin{lem}\label{prel1}
Let $I$ be a compact interval in $\mathbb{R}$.
For each $\epsilon>0$ there is a $\delta>0$ such that
for any measurable set $J\subset I$, the following holds. 
For each element $g\in B_{\delta}(1_{\bf{F}})$, $$(1-\epsilon)\mu(J)<\mu(J\cdot g)<(1+\epsilon)\mu(J)$$
\end{lem}

\section{The paradox} 

Throughout this section, we fix $I$ to be a compact interval in $\mathbb{R}$.
We first prove a basic Lemma about elements of $P_f(H)$.

\begin{lem}\label{basiclemma}
There is a $\delta>0$ such that for each $u\in P_f(H)$ and $g_1,...,g_6\in B_{\delta}(1_{\mathbf{F}})$, the following conditions hold.
\begin{enumerate}

\item Assume that $\phi^{-1}_{u,a}[\frac{1}{2},1]\cap I$ has measure at least $\frac{1}{2}\mu(I)$.
Then there is a set of positive measure $L\subseteq I$ and distinct elements $h_1,...,h_4\in \{1_{\mathbf{F}},g_1,...,g_6\}$
such that $L\cdot h_i\subseteq \phi^{-1}_{u,a}[\frac{1}{2},1]$ for each $1\leq i\leq 4$.

\item Assume that $\phi^{-1}_{u,b}[\frac{1}{2},1]\cap I$ has measure at least $\frac{1}{2}\mu(I)$.
Then there is a set of positive measure $L\subseteq I$ and distinct elements $h_1,...,h_4\in \{1_{\mathbf{F}},g_1,...,g_6\}$
such that $L\cdot h_i\subseteq \phi^{-1}_{u,b}[\frac{1}{2},1]$ for each $1\leq i\leq 4$.

\end{enumerate}
\end{lem}

\begin{proof}
Let $I'$ be an interval properly containing $I$ such that $\mu(I'\setminus I)<\frac{1}{48}\mu(I)$.
We choose a $\delta>0$ such that the following conditions are satisfied.
\begin{enumerate}
\item For each $g\in B_{\delta}(1_H)$ and $J\subset I$ it holds that $(1-\frac{1}{48})\mu(J)<\mu(J\cdot g^{-1})<(1+\frac{1}{48})\mu(J)$.
\item For each $g\in B_{\delta}(1_H)$ it holds that $I\cdot g^{-1}\subseteq I'$
\end{enumerate}
A positive real satisfying the first condition is obtained from Lemma \ref{prel1}
together with continuity of group inversion.
A positive real that satisfies the second is obtained from the continuity of the action of $\bf{F}$ on $\mathbb{R}\cup \{\infty\}$.
We claim that the smaller of these two reals satisfies the statement of the Lemma.

The sets $\phi_{u,a}^{-1}[\frac{1}{2},1], \phi_{u,b}^{-1}[\frac{1}{2},1]$
cover $I$, so either $\phi_{u,a}^{-1}[\frac{1}{2},1]\cap I$ or $\phi_{u,b}^{-1}[\frac{1}{2},1]\cap I$
has measure at least $\frac{1}{2}\mu(I)$.
Let us assume that the set $J=\phi_{u,a}^{-1}[\frac{1}{2},1]\cap I$ 
has measure at least $\frac{1}{2}\mu(I)$.

Fix $g_1,...,g_6\in B_{\delta}(1_{\bf{F}})$ and $g_0=1_{\mathbf{F}}$,
and define the sets $L_i=(J\cdot g_i^{-1})\cap I$.
From our assumptions above,
it follows that $\mu(L_i)\geq \frac{1}{2}\mu(I)-\frac{1}{24}\mu(I)$
for $1\leq i\leq 6$ and $\mu(L_0)\geq \frac{1}{2}\mu(I)$.

Now consider the function $f=\sum_{0\leq i\leq 6} \chi_{L_i}$.
This function has support in $I$ and 
$$\int \sum_{0\leq i\leq 6} \chi_{L_i}d\mu=\sum_{0\leq i\leq 6}\mu(L_i)\geq \frac{7}{2}\mu(I)-\frac{1}{4}\mu(I)>3\mu(I)$$
Therefore, it holds that the function $f$ takes value at least $4$
on a positive measure subset of $I$.
(Or else, $\int f d\mu\leq 3\mu(I)$.)

By the pigeonhole principle, there is a positive measure set $L\subseteq I$
and a four element set $\{h_1,h_2,h_3,h_4\}\subset \{1_{\mathbf{F}},g_1,...,g_6\}$
such that $L\cdot h_i\subseteq  \phi^{-1}_{u,a}[\frac{1}{2},1]$ for each $1\leq i\leq 4$.
This proves the Lemma.
\end{proof}

We choose a set of reduced words $T=\{g_1,...,g_6\}\cup \{h_1,...,h_6\}$ in $\mathbf{F}=\langle a,b\rangle$ such that:
\begin{enumerate}
\item Each $g_i$ is a reduced word of the form $a^i w a^{-i}$ in $\langle a,b\rangle$. 
\item Each $h_i$ is a reduced word of the form $b^i w' b^{-i}$ in $\langle a,b\rangle$. 
\item $T=\{g_1,...,g_6\}\cup \{h_1,...,h_6\}\subset B_{\delta}(1_{\mathbf{F}})$.
\item The restriction of each element in $T$
to $I$ has bounded image in $\mathbb{R}$, i.e. $I$ does not contain the singularity points
of these elements.
\end{enumerate}
The existence of such words is guaranteed by continuity of conjugation and multiplication in the non discrete free group $\mathbf{F}$.
We leave this as an exercise for the reader.
We remark that condition $4$ is satisfied by elements of $B_{\delta}(1_{\mathbf{F}})$, where $\delta$ is from the proof of \ref{basiclemma}.

We find a set of elements $\tilde{T}=\{\mathbf{g}_1,...,\mathbf{g}_6\}\cup \{\mathbf{h_1},...,\mathbf{h}_6\}$
in $H$ such that the action of $\mathbf{g}_i, \mathbf{h}_i$ agrees respectively with the action of $g_i,h_i$
on $I$.
The fact that such elements can be found has been made explicit in the proof of non amenability for the groups in
\cite{Monod} and \cite{LodhaMoore}.

The set $\tilde{T}\cup \{1_H\}$ will be the set of translating elements of our paradox.
We now prove that this is indeed a translating set by demonstrating that it satisfies Hall's $2$-marriage condition.

\begin{prop}\label{Marriage}
For each $u\in P_f(H)$ it holds that $|(\tilde{T}\cup \{1_H\}) \cdot u|\geq 2|u|$.
\end{prop}

\begin{proof}
Assume without loss of generality that $\mu(\phi^{-1}_{u,a}[\frac{1}{2},1]\cap I)\geq \frac{1}{2}\mu(I)$.
Following the statement of \ref{basiclemma} applied to the elements $\{h_1,...,h_6\}$, there is a four element set $\{k_1,k_2,k_3,k_4\}\subseteq \{1_{\bf{F}},h_1,...,h_6\}$,
and a set of positive measure $I'\subseteq I$,
such that for each $k\in \{k_1,k_2,k_3,k_4\}$ we have that $I'\cdot k\subseteq \phi^{-1}_{u,a}[\frac{1}{2},1]$.
By discarding a null set if necessary, we can assume that $\mathbf{M}$ acts freely on the orbits that intersect $I'$.
For notational convenience, we denote as $\mathbf{k}_i$ the element of $\{\mathbf{h}_1,...,\mathbf{h}_6,1_{H}\}$
that agrees with $k_i$ on $I$.

Fix $x\in I'$.
Define the sets 
$$L_i=\{g\in H\mid x\cdot g\in A_{x\cdot k_i}\}=\{g\in H\mid x\cdot g\in A_{x\cdot \mathbf{k}_i}\}$$
$$=\{g\in H\mid x\cdot g\in x\cdot (k_i\cdot P_a)\}$$
for $1\leq i\leq 4$.
By definition of $k_i$ and from the freeness of the action of $\mathbf{M}$ on the orbit of $x$, it follows that 
$x\cdot (k_1\cdot P_a),...,x\cdot (k_4\cdot P_a)$ are pairwise disjoint, and so $L_1,...,L_4$ are pairwise disjoint in $H$.

By our assumption, we know that 
$\phi_{u,a}(x\cdot k_i)\geq \frac{1}{2}$ for each $1\leq i\leq 4$.
Therefore,
$$\phi_{u,a}(x\cdot k_i)=\phi_{u,a}(x\cdot \mathbf{k}_i)=\lambda_{u,x\cdot \mathbf{k}_i}(A_{x\cdot \mathbf{k}_i})$$
$$=u(\{g\in H\mid (x\cdot \mathbf{k}_i)\cdot g\in A_{x\cdot \mathbf{k}_i}\})=u(\mathbf{k}_i^{-1}\cdot L_i)\geq \frac{1}{2}$$

We collect the main observations:
\begin{enumerate}
\item $L_1,...,L_4$ are pairwise disjoint.
\item For each $1\leq i\leq 4$, $\frac{|u\cap (\mathbf{k}_i^{-1}\cdot L_i)|}{|u|}\geq \frac{1}{2}$.
\end{enumerate}
It follows that $|\{\mathbf{k}_1,...,\mathbf{k}_4\}\cdot u|\geq 2|u|$.
\end{proof}

{\bf Proof of the main theorem}:
Consider the sets $S_1=\tilde{T}\cup \{1_H\}, S_2=\tilde{T}$.
It follows that for each $u_1,u_2\in P_f(H)$
$$|S_1\cdot u_1 \cup S_2\cdot u_2|=|(\tilde{T}\cdot (u_1\cup u_2)) \cup u_1|\geq |u_1\cup u_2|+|u_1\cap u_2|=|u_1|+|u_2|$$
It follows from Hall's marriage theorem for evenly colored $2$-marriages (see Theorem $2.6$ \cite{ErshovGolanSapir})
that $S_1,S_2$ are translating sets for a paradoxical decomposition of $H$.
It follows that the Tarski number of $H$ is at most $25$.

\end{document}